\documentclass[11pt]{amsart}
\usepackage{amssymb}
\usepackage{amsmath,amssymb,amsfonts}

\newtheorem{theorem}{Theorem}[section]
\newtheorem{corollary}[theorem]{Corollary}

\newtheorem{proposition}[theorem]{Proposition}

\newtheorem{remark}{Remark}[section]

\newtheorem{definition}{Definition}[section]

\begin{document}
\title[D'Atri spaces of type $k$]{D'Atri spaces of type $k$ and related classes of geometries concerning Jacobi operators}
\author{Teresa Arias-Marco}
\address{Departamento de Matem\'{a}ticas, Universidad de Extremadura, 06006 Badajoz,
Spain }
\email{ariasmarco@unex.es}
\author{Maria J. Druetta}
\address{CIEM - FaMAF, Universidad Nacional de C\'{o}rdoba, 5000 C\'{o}rdoba,
Argentina}
\email{druetta@famaf.unc.edu.ar}
\keywords{D'Atri spaces, $k$-D'Atri spaces, Ledger's recursion formula, curvature
invariants, geodesic spheres, shape operator, $k$-th elementary symmetric
functions, spaces of Iwasawa type, $\frak{SC}$-spaces}
\subjclass[2010]{53C21, 53C25, 53C30, 22E99}
\thanks{The authors were partially supported by CONICET, FONCyT and SECyT (UNC). The
first author's work has also been supported by D.G.I.~(Spain) and FEDER
Project MTM2010-15444, by Junta de Extremadura and FEDER funds, by DFG
Sonderforschungsbereich~647 and the program ``Becas Iberoam\'{e}rica.
J\'{o}venes Profesores e Investigadores. Santander Universidades'' of
Santander Bank.}

\maketitle

\begin{abstract}
In this article we continue the study of the geometry of $k$-D'Atri spaces, $%
1\leq k$ $\leq n-1$ ($n$ denotes the dimension of the manifold)$,$ began by
the second author. It is known that $k$-D'Atri spaces, $k\geq 1,$ are
related to properties of Jacobi operators $R_{v}$ along geodesics, since she has shown that
${\operatorname{tr}}R_{v}$, ${\operatorname{tr}}R_{v}^{2}$ are invariant under
the geodesic flow for any unit tangent vector $v$. Here, assuming that the
Riemannian manifold is a D'Atri space, we prove in our main result that ${%
\operatorname{tr}}R_{v}^{3}$ is also invariant under the geodesic flow if $%
k\geq 3$. In addition, other properties of Jacobi operators related to the
Ledger conditions are obtained and they are used to give applications to
Iwasawa type spaces. In the class of D'Atri spaces of Iwasawa type, we show
two different characterizations of the symmetric spaces of noncompact type:
they are exactly the $\frak{C}$-spaces and on the other hand they are $k$%
-D'Atri spaces for some $k\geq 3.$ In the last case, they are $k$-D'Atri for
all $k=1,...,n-1$ as well. In particular, Damek-Ricci spaces that are $k$%
-D'Atri for some $k\geq 3$ are symmetric.

Finally, we characterize $k$-D'Atri spaces for all $k=1,...,n-1$ as the $%
\frak{SC}$-spaces (geodesic symmetries preserve the principal curvatures of
small geodesic spheres). Moreover, applying this result in the case of $4$%
-dimensional homogeneous spaces we prove that the properties of being a
D'Atri ($1$-D'Atri) space, or a $3$-D'Atri space, are equivalent to the
property of being a $k$-D'Atri space for all $k=1,2,3$.
\end{abstract}

\section{Introduction and Preliminaries}\label{sec:intro}

Let $M$ be a $n$-dimensional Riemannian manifold, $\nabla $ the Levi Civita
connection and let $R$ denote the associated curvature tensor defined by $%
R(u,v)=[\nabla _{u},\nabla _{v}]-\nabla _{[u,v]}$ for all $u,v\in TM.$ If
$\left| v\right| =1$, the Jacobi operators $R_{v}$ are defined by $R_{v}w=R(w,v)v.$

Let $m\in M$ be a fixed point and $v\in T_{m}M,$ $\left| v\right| =1;$ we
denote by $\gamma _{v}(t)$ the geodesic in $M$ with $\gamma _{v}(0)=m\,$ and
$\gamma _{v}^{\prime }(0)=v.$ Note that $\exp _{m}tv=\gamma _{v}(t)$
whenever $\exp _{m}$, the geometric exponential map of $M$, is defined.
Moreover, for each small $t>0,$ we denote by $S_{v}(t)$ the shape operator
(with respect to the outward unit normal field $\gamma _{v}^{\prime }(t)$)
of the geodesic sphere
\begin{equation*}
G_{m}(t)=\{\gamma _{w}(t)=\exp _{m}(tw):w\in T_{m}M,\,\left| w\right| =1\}
\end{equation*}
at $\gamma _{v}(t).$ By definition, for each $m\in M$ the geodesic
symmetries $s_{m}$ are locally defined by
\begin{equation*}
s_{m}=\exp _{m}\circ \sigma _{0}\circ \exp _{m}^{-1},\text{ where }\sigma
_{0}=-\text{Id.}
\end{equation*}
Equivalently, $s_{m}(p)=\exp _{m}(-\exp _{m}^{-1}(p))\,$ for all $p\in M$
where $\exp _{m}$ is locally defined as a diffeomorphism, or $s_{m}(\gamma
_{v}(t))=\gamma _{v}(-t)$ for all real $t\sim 0$.

D'Atri spaces were introduced by J.~E.~D'Atri and H.~K.~Nickerson in \cite
{D'A-N.69}. $M$ is called a \emph{D'Atri space} if the local geodesic
symmetries are volume-preserving (i.e. they preserve the volume element up
to a sign). An equivalent definition is given by the condition that the
geodesic symmetries preserve the mean curvature of small geodesic spheres;
that is ${\operatorname{tr}}S_{v}(t)={\operatorname{tr}}S_{-v}(t)$.
Obviously, D'Atri spaces are a natural generalization of locally symmetric
spaces (where the local geodesic symmetries are isometries) and in dimension
two are locally symmetric, so they have constant sectional curvature. The
third dimensional classification was done by O.~Kowalski in \cite{Ko.83}
where he proved that all of them are either locally symmetric or locally
isometric to a naturally reductive space. See \cite{KPV} for references
about D'Atri spaces and related topics.

Many characterizations of D'Atri spaces exist but the most relevant for our
work was proved by J.~E.~D'Atri and H.~K.~Nickerson \cite{D'A-N.69} and it
was improved by Z.~I.~Szab\'{o}~\cite{Sz.93}; namely, $M$ is a D'Atri space
if and only if it satisfies the series of all odd Ledger conditions~$
L_{2k+1}=0$, $k\ge 1$. The \emph{Ledger conditions} are an infinite series
of curvature conditions derived from the so-called Ledger recurrence
formula, which nowadays, have become of a special and important relevance
(see \cite{P-T}, \cite{AM09}). For example, Z.~I.~Szab\'{o}~\cite{Sz.93}
proved that $L_{3}=0$ implies that the manifold is real analytic. Moreover,
the first author and O. Kowalski \cite{AM-Ko.08} classified the $4$-dimensional homogeneous Riemannian manifolds which satisfy $L_{3}=0$ and
used this result to classify the $4$-dimensional homogeneous D'Atri spaces,
as well (see also \cite{AM}, \cite{AM.07}).

In Section \ref{sec:Ledger} of this work we study properties of the Jacobi
operators along geodesics related to Ledger's conditions as $L_{3}=0,$
$L_{5}=0,$ $L_{7}=0$, which play an important role to prove two of our main
results, Theorem \ref{th:k-D'A property} and Theorem \ref{th:Ledger+Iwasawa}, developed in Section \ref{sec:k-DAtri} and Section \ref{sec:kDatri and
Iwasawa}, respectively.

$M$ is called a D'Atri space of type $k$ or a $k$-D'Atri space, $1\leq k\leq
n-1,$ if the geodesic symmetries preserve the $k$-th elementary symmetric
functions of the eigenvalues of the shape operators of all small geodesic
spheres. Recall, that the $k$-th elementary symmetric functions $\sigma
_{k}, $ $k=1,...,n,$ of the eigenvalues of a symmetric endomorphism $A$ on a
$n$-dimensional real vector space are determined by its characteristic
polynomial as follows,
\begin{equation*}
\text{det}(\lambda I-A)=\lambda ^{n}-\sigma _{1}(A)\lambda
^{n-1}+....+(-1)^{k}\sigma _{k}(A)\lambda ^{n-k}+...+(-1)^{n}\sigma _{n}(A),
\end{equation*}
\begin{equation*}
\text{\thinspace }\sigma _{k}(A)=\sum_{i_{1}<i_{2}<....<i_{k}}\lambda
_{i_{1}}(A)\cdot \cdot \cdot \lambda _{i_{k}}(A)
\end{equation*}
with $1\leq i_{1}<i_{2}<...<i_{k}\leq n$ and $\{\lambda _{1}(A),...,\lambda
_{n}(A)\}$ the set of $n$ eigenvalues of $A.$ Thus, $M\,$\thinspace
\thinspace is a $k$-D'Atri space if and only if for each small $r>0$
\begin{equation*}
\sigma _{k}(S_{v}(r))=\sigma _{k}(S_{-v}(r))\text{ for all unit vector }v\in
T_{m}M,
\end{equation*}
where $S_{\pm v}(r)$ denotes the shape operators of $G_{m}(r)$ at the points
$\exp _{m}r(\pm v)$. Therefore, the $1$-D'Atri property is obviously the
D'Atri condition.

D'Atri space of type $k$ definitions were introduced by O. Kowalski, F.
Pr\"{u}fer and L. Vanhecke in \cite{KPV} as a natural analogues of the
concept of D'Atri space and it was started as open problem to analyze if all
these analog notions are equivalent or not. The first attempt to solve this
problem has been done by the second author in \cite{D10}, where it was shown
that the notions of D'Atri spaces ($1$-D'Atri) and $2$-D'Atri spaces are
equivalent. Now, we continue such study in Section \ref{sec:kDatri and
Iwasawa} and Section \ref{sec:GCandk-D'Atri} considering Iwasawa type spaces
and 4-dimensional homogeneous spaces, respectively.

Besides, it is also shown in \cite{D10} that $k$-D'Atri spaces, $k\geq 1,$
are related to properties of Jacobi operators as the invariance under the
geodesic flow of ${\operatorname{tr}}R_{v}$ and ${\operatorname{tr}}
R_{v}^{2},$ respectively. In Section \ref{sec:k-DAtri} we complete this fact
(see Proposition \ref{pro:2.2D10} and Proposition \ref{pro:correction}) and obtain
the same result for ${\operatorname{tr}}R_{v}^{3}$ in Theorem \ref{th:k-D'A
property}, under the assumption that $M$ is also a D'Atri space. Note that
throughout the paper we can assume $n\geq 3$ and $k\geq 2$.

One of the consequences of Theorem \ref{th:k-D'A property} is obtained in
Section \ref{sec:kDatri and Iwasawa}, considering spaces of Iwasawa type,
where the symmetric spaces are characterized as D'Atri spaces which are $k$
-D'Atri for some $k\geq 3$. Some properties of the $k$-D'Atri condition $
(k\geq 1)$ in the class of Iwasawa type spaces have been study in \cite{D09}
and \cite{D10}. In particular, the symmetric ones were characterized as the $
k$-D'Atri spaces for all $k=1,...,n-1.$ Here, we continue such study
proving that every D'Atri space satisfying that ${\operatorname{tr}}
R_{v}^{3} $ is invariant under the geodesic flow is a symmetric space.
Moreover, we also get that D'Atri spaces of Iwasawa type which are also ${
\frak{C}}$-spaces are symmetric.

$\mathfrak{C}$-spaces were introduced by J.~Berndt and L.Vanhecke in \cite
{B-V.92}. By definition, $M$ is a \emph{$\mathfrak{C}$-space} if for each
geodesic $\gamma ,$ the eigenvalues of $R_{\gamma _{v}^{\prime }(t)}$ are
constant along $\gamma _{v}(t)$. For locally symmetric spaces this is always
the case, so $\mathfrak{C}$-spaces are another natural generalization of
locally symmetric spaces. In the last section we describe $\frak{C}$-spaces
as those whose geodesic symmetries preserve the eigenvalues of Jacobi
operators (Proposition \ref{pro:C-cha}).
In the case of Iwasawa type spaces of rank one it was shown in \cite{D01}
that $\frak{C}$-spaces are symmetric. Moreover, Damek-Ricci spaces are rank one
spaces of Iwasawa type and the non-symmetric ones were the first examples
of D'Atri spaces which are not $\mathfrak{C}$-spaces~\cite{B-Tr-V}. However,
it is an open question whether a $\mathfrak{C}$-space is a D'Atri space.

In Section \ref{sec:GCandk-D'Atri} we will characterize the $k$-D'Atri
spaces for all $k\geq 1$ as the ${\frak{SC}}$-spaces (Theorem \ref{th:k-D'A
is GC}). Thus, we complete Theorem 2.6 of \cite{D10} where it was proved
that $k$-D'Atri spaces for all $k=1,...,n-1$ are ${\frak{C}}$-spaces. $M$
is a ${\frak{SC}}$\emph{-space} if the geodesic symmetries $s_{m}$ preserve
the eigenvalues (also called principal curvatures) of the shape operators $S_{v}(t)$ of small geodesic spheres centered at $m$ for all $m\in M.$ See
\cite{B-P-V} to know more about this kind of spaces. Then, as a consequence
of this characterization, we prove that the $3$-D'Atri condition is also an
equivalent notion to the D'Atri and $2$-D'Atri conditions for the $4$-dimensional homogeneous case.

\section{Ledger's conditions and properties of Jacobi operators}

\label{sec:Ledger}

Let $M$ be a Riemannian manifold and $m\in M$ be a fix point. Let $v\in
T_{m}M$ be a unit vector and consider a small real $r>0.$ If $M$ is real
analytic, then it is well known that the endomorphism $C_{v}(r)=rS_{v}(r)=
\sum_{k=0}^{\infty }\alpha _{k}(v)r^{k}$ with $\alpha _{k}(v)=\frac{1}{k!}
C_{v}^{(k)}(0)$, gives the power series expansion of $C_{v}(r)$ at $r=0,$
where
\begin{equation*}
C_{v}^{(k)}(0)=\frac{D^{k}}{dr^{k}}\left. C_{v}(r)\right| _{r=0}\,,\text{ \ }
k\geq 1,
\end{equation*}
may be computed by using the recursion formula of Ledger that is given by
\begin{equation}
(k+1)C_{v}^{(k)}(0)=-k(k-1)R_{v}^{(k-2)}-\sum_{l=2}^{k-2}\tbinom{k}{l}
C_{v}^{(l)}(0)C_{v}^{(k-l)}(0)\text{ for }k\geq 2,  \label{eq:Ledger}
\end{equation}
with $C_{v}(0)={\operatorname{Id}}$, $C_{v}^{\prime }(0)=0$ (see \cite{Be,
CV, KPV}). Here we use the notation $R_{v}=R_{v}(0)$ and $
R_{v}^{(k)}=R_{v}^{(k)}(0)$, $k\geq 1,$ the $k$-th covariant derivative of
the tensor $R_{\gamma _{v}^{^{\prime }}(t)}$ along $\gamma _{v}$ at $t=0.$
Then,
\begin{eqnarray*}
R_{v}^{\prime } &=&R_{v}^{\prime }(0)=\left. \left( \nabla _{\gamma
_{v}^{\prime }(t)}R_{\gamma _{v}^{\prime }(t)}\right) \right|
_{t=0}\,\,\,\,\,\,\,\,\,\text{and} \\
R_{v}^{(k)} &=&R_{v}^{(k)}(0)=\left. \left( \nabla _{\gamma _{v}^{\prime
}(t)}R_{\gamma _{v}^{\prime }(t)}^{(k-1)}\right) \right| _{t=0}\text{ for
all }k\geq 2.
\end{eqnarray*}
Thus, from formula \eqref{eq:Ledger} we have
\begin{equation}
\begin{split}
C_{v}(0)=& \,{\operatorname{Id}},\quad C_{v}^{\prime }(0)=0,\quad
C_{v}^{\prime \prime }(0)=-\tfrac{2}{3}R_{v},\quad C_{v}^{(3)}(0)=-\tfrac{3}{
2}R_{v}^{\prime }, \\
C_{v}^{(4)}(0)=& -\tfrac{4}{5}\left( 3R_{v}^{\prime \prime }+\tfrac{2}{3}
R_{v}\circ R_{v}\right) , \\
C_{v}^{(5)}(0)=& -\tfrac{5}{3}\left( 2R_{v}^{(3)}+R_{v}^{\prime }\circ
R_{v}+R_{v}\circ R_{v}^{\prime }\right) , \\
C_{v}^{(6)}(0)=& -\tfrac{3}{7}\left( 10R_{v}^{(4)}+8R_{v}\circ R_{v}^{\prime
\prime }+8R_{v}^{\prime \prime }\circ R_{v}+15R_{v}^{\prime }\circ
R_{v}^{\prime }+\tfrac{32}{9}R_{v}\circ R_{v}\circ R_{v}\right) , \\
C_{v}^{(7)}(0)=& -\tfrac{7}{12}\left( 9R_{v}^{(5)}+10R_{v}\circ
R_{v}^{(3)}+10R_{v}^{(3)}\circ R_{v}+27R_{v}^{\prime }\circ R_{v}^{\prime
\prime }+27R_{v}^{\prime \prime }\circ R_{v}^{\prime }\right. \\
& \left. \quad \quad \quad +11R_{v}\circ R_{v}\circ R_{v}^{\prime
}+11R_{v}^{\prime }\circ R_{v}\circ R_{v}+10R_{v}\circ R_{v}^{\prime }\circ
R_{v}\right).
\end{split}
\label{eq:opCparticular}
\end{equation}
On the other hand, the Ledger conditions are defined in terms of $
C_{v}^{(k)}(0)$, $k\geq 1$, and the well-known characterization of D'Atri
spaces is given using those conditions of odd order. In the previous context
we have,

\begin{definition}
\label{def:Ledger} If for each unit vector $v\in T_{m}M$
\begin{equation*}
L_{k}={\operatorname{tr}}C_{v}^{(k)}(0)=\frac{d^{k}}{dr^{k}}\left. {
\operatorname{tr}}C_{v}(r)\right| _{r=0}\,,\,k\geq 1,
\end{equation*}
then $L_{2k+1}=0$ and $L_{2k}=\,c_{2k}\,$, $k\geq 1,$ define the Ledger
conditions (associated to $v$) of odd order and even order, respectively, at
the point $m$.
\end{definition}

\begin{remark}
\label{rem:D'A iff Ledger} It is well-known that M is a D'Atri space if and
only if the infinite series of Ledger conditions of odd order are satisfied.
\end{remark}

In the rest of this section, we will show how to use the three first odd
Ledger conditions to obtain some useful identities.

\begin{remark}
\label{rem:Der1}If $v\in T_{m}M$ is a unit vector, then for any geodesic $
\gamma _{v}(t)$
\begin{equation*}
\frac{d}{dt}\left( {\operatorname{tr}}(R_{\gamma _{v}^{\prime
}(t)}^{(k)}\circ R_{\gamma _{v}^{\prime }(t)}^{(k)})\right) =2{\operatorname{tr}}\left( R_{\gamma _{v}^{\prime }(t)}^{(k)}\circ R_{\gamma
_{v}^{\prime }(t)}^{(k+1)}\right) \text{ for all }k\geq 0.
\end{equation*}
\end{remark}

In fact, for any orthonormal parallel basis $\{e_{i}(t)\}_{i=1}^{n}$ along
$\gamma _{v}(t)$ with $e_{n}(t)=\gamma _{v}^{\prime }(t)$, using that all
operators $R_{\gamma _{v}^{\prime}(t)}^{(k)}$ are symmetric, we have
\begin{equation*}
\begin{split}
\frac{d}{dt}\left( {\operatorname{tr}}\left( R_{\gamma _{v}^{\prime
}(t)}^{(k)}\right) ^{2}\right) =& \sum_{i=1}^{n-1}\frac{d}{dt}\left\langle
\left( R_{\gamma _{v}^{\prime}(t)}^{(k)}\right)
^{2}e_{i}(t),e_{i}(t)\right\rangle \\
=& \sum_{i=1}^{n-1}\frac{d}{dt}\left\langle R_{\gamma _{v}^{\prime
}(t)}^{(k)}e_{i}(t),R_{\gamma _{v}^{\prime}(t)}^{(k)}e_{i}(t)\right\rangle \\
=& 2\sum_{i=1}^{n-1}\left\langle \nabla _{\gamma _{v}^{\prime }(t)}\left(
R_{\gamma _{v}^{\prime }(t)}^{(k)}e_{i}(t)\right) ,R_{\gamma _{v}^{\prime
}(t)}^{(k)}e_{i}(t)\right\rangle \\
=& 2\sum_{i=1}^{n-1}\left\langle \left( \nabla _{\gamma _{v}^{\prime
}(t)}R_{\gamma _{v}^{\prime }(t)}^{(k)}\right) e_{i}(t),R_{\gamma
_{v}^{\prime }(t)}^{(k)}e_{i}(t)\right\rangle \\
=& 2\sum_{i=1}^{n-1}\left\langle R_{\gamma _{v}^{\prime
}(t)}^{(k+1)}e_{i}(t),R_{\gamma _{v}^{\prime
}(t)}^{(k)}e_{i}(t)\right\rangle =2{\operatorname{tr}}\left( R_{\gamma
_{v}^{\prime }(t)}^{(k)}\circ R_{\gamma _{v}(t)}^{(k+1)}\right) .
\end{split}
\end{equation*}

\begin{proposition}
\label{pro:L5} If $L_{3}=0$, then

\begin{itemize}
\item[(i)]  ${\operatorname{tr}}R_{v}^{(k)}=0$ for all $k\geq 1$.
Consequently, ${\operatorname{tr}}R_{\gamma _{v}^{\prime}(t)}^{(k)}=0$
along $\gamma _{v}(t)$ and ${\operatorname{tr}}(R_{v}^{(k-1)})$ is invariant
under the geodesic (local) flow for all $k\geq 1$.

\item[(ii)]  The second odd Ledger condition $L_{5}=0$ becomes
\begin{equation}
{\operatorname{tr}}(R_{v}\circ R_{v}^{\prime })=0\text{ for all unit vector }
v\in T_{m}M.  \label{eq:L5}
\end{equation}
Equivalently, ${\operatorname{tr}}(R_{v}^{2})$ is invariant under the
geodesic (local) flow.
\end{itemize}
\end{proposition}

\begin{proof}
It is contained in the proof of \cite[Proposition 2.2]{D10}
having into account that we use the facts $L_{3}={\operatorname{tr}}
R_{v}^{\prime }=0$ and $L_{5}={\operatorname{tr}}(R_{v}\circ R_{v}^{\prime
})=0$ (not their proofs). See (ii) and the last part of (iii) in the proof
of such proposition. Note that from \eqref{eq:opCparticular} we express
\begin{equation*}
L_{5}=-\tfrac{5}{3}{\operatorname{tr}}(2R_{v}^{(3)}+R_{v}^{\prime }\circ
R_{v}+R_{v}\circ R_{v}^{\prime })=-\tfrac{10}{3}\left( {\operatorname{tr}}
\left( R_{v}\circ R_{v}^{\prime }\right) \right).
\end{equation*}
\end{proof}

\begin{proposition}
\label{pro:Der2} If $v\in T_{m}M$ is a unit vector, then for any geodesic $\gamma _{v}(t)$
\begin{equation*}
\frac{d}{dt}\left( {\operatorname{tr}}\left( R_{\gamma _{v}^{\prime}(t)}^{(k)}\circ R_{\gamma _{v}^{\prime}(t)}\right) \right) ={\operatorname{tr}}\left( R_{\gamma _{v}^{\prime }(t)}^{(k+1)}\circ R_{\gamma
_{v}^{\prime }(t)}+R_{\gamma _{v}^{\prime }(t)}^{(k)}\circ R_{\gamma
_{v}^{\prime }(t)}^{\prime }\right) ,\;k\geq 1.
\end{equation*}
\end{proposition}

\begin{proof} \ Let $\{e_{i}(t)\}_{i=1}^{n}$ an orthonormal parallel
basis along $\gamma _{v}(t)$ with $e_{n}(t)=\gamma _{v}^{\prime }(t)$. We
compute
\begin{equation*}
\begin{split}
& \frac{d}{dt}\left( {\operatorname{tr}}\left( R_{\gamma _{v}^{\prime}(t)}^{(k)}\circ R_{\gamma _{v}^{\prime}(t)}\right) \right)
=\sum_{i=1}^{n-1}\frac{d}{dt}\left\langle \left( R_{\gamma _{v}^{\prime}(t)}^{(k)}\circ R_{\gamma _{v}^{\prime}(t)}\right)
e_{i}(t),e_{i}(t)\right\rangle \\
& =\sum_{i=1}^{n-1}\frac{d}{dt}\left\langle R_{\gamma _{v}^{\prime}(t)}e_{i}(t),R_{\gamma _{v}^{\prime}(t)}^{(k)}e_{i}(t)\right\rangle
\end{split}
\end{equation*}
\begin{equation*}
\begin{split}
& =\sum_{i=1}^{n-1}\left\{ \left\langle \nabla _{\gamma _{v}^{\prime
}(t)}\left( R_{\gamma _{v}^{\prime }(t)}e_{i}(t)\right) ,R_{\gamma
_{v}^{\prime }(t)}^{(k)}e_{i}(t)\right\rangle +\left\langle R_{\gamma
_{v}^{\prime }(t)}e_{i}(t),\nabla _{\gamma _{v}^{\prime }(t)}\left(
R_{\gamma _{v}^{\prime }(t)}^{(k)}e_{i}(t)\right) \right\rangle \right\} \\
& =\sum_{i=1}^{n-1}\left\{ \left\langle \left( \nabla _{\gamma _{v}^{\prime
}(t)}R_{\gamma _{v}^{\prime }(t)}\right) e_{i}(t),R_{\gamma _{v}^{\prime
}(t)}^{(k)}e_{i}(t)\right\rangle +\left\langle R_{\gamma _{v}^{\prime
}(t)}e_{i}(t),\left( \nabla _{\gamma _{v}^{\prime }(t)}R_{\gamma
_{v}^{\prime }(t)}^{(k)}\right) e_{i}(t)\right\rangle \right\} \\
& =\sum_{i=1}^{n-1}\left\{ \left\langle R_{\gamma _{v}^{\prime }(t)}^{\prime
}e_{i}(t),R_{\gamma _{v}^{\prime }(t)}^{(k)}e_{i}(t)\right\rangle
+\left\langle R_{\gamma _{v}^{\prime }(t)}e_{i}(t),R_{\gamma _{v}^{\prime
}(t)}^{(k+1)}e_{i}(t)\right\rangle \right\} \\
& =\sum_{i=1}^{n-1}\left\{ \left\langle \left( R_{\gamma _{v}^{\prime
}(t)}^{(k)}\circ R_{\gamma _{v}^{\prime }(t)}^{\prime }\right)
e_{i}(t),e_{i}(t)\right\rangle +\left\langle \left( R_{\gamma _{v}^{\prime
}(t)}^{(k+1)}\circ R_{\gamma _{v}^{\prime }(t)}\right)
e_{i}(t),e_{i}(t)\right\rangle \right\} \\
& ={\operatorname{tr}}\left( R_{\gamma _{v}^{\prime }(t)}^{(k+1)}\circ
R_{\gamma _{v}^{\prime }(t)}+R_{\gamma _{v}^{\prime }(t)}^{(k)}\circ
R_{\gamma _{v}^{\prime }(t)}^{\prime }\right).
\end{split}
\end{equation*}
\end{proof}

\begin{proposition}
\label{pro:L7} If $L_{3}=0$ and $L_{5}=0$, then the third odd Ledger's
condition $L_{7}=0$ becomes
\begin{equation}
16{\operatorname{tr}}(R_{v}^{\prime }\circ R_{v}^{2})-3{\operatorname{tr}}
(R_{v}^{\prime }\circ R_{v}^{\prime \prime })=0\text{ for all unit vector }
v\in T_{m}M.  \label{eq:L7}
\end{equation}
Equivalently, $L_{7}=0$ if and only if ${\operatorname{tr}}
(32R_{v}^{3}-9R_{v}^{\prime }\circ R_{v}^{\prime })$ is invariant under the
geodesic flow.
\end{proposition}

\begin{proof} \ Let $v\in T_{m}M$ be a unit vector. We first show that if
the property ${\operatorname{tr}}(R_{v}\circ R_{v}^{\prime })=0$ is
fulfilled, then
\begin{equation}
\begin{split}
{\operatorname{tr}}(R_{\gamma _{v}^{\prime }(t)}^{\prime }\circ R_{\gamma
_{v}^{\prime }(t)}^{\prime }+R_{\gamma _{v}^{\prime }(t)}\circ R_{\gamma
_{v}^{\prime }(t)}^{\prime \prime })& =0, \\
{\operatorname{tr}}(3R_{\gamma _{v}^{\prime }(t)}^{\prime \prime }\circ
R_{\gamma _{v}^{\prime }(t)}^{\prime }+R_{\gamma _{v}^{\prime }(t)}\circ
R_{\gamma _{v}^{\prime }(t)}^{(3)})& =0
\end{split}
\label{eq:derL5t}
\end{equation}
along $\gamma _{v}(t)$. Equivalently,
\begin{equation}
\begin{split}
{\operatorname{tr}}(R_{v}^{\prime }\circ R_{v}^{\prime }+R_{v}\circ
R_{v}^{\prime \prime })& =0, \\
{\operatorname{tr}}(3R_{v}^{\prime \prime }\circ R_{v}^{\prime }+R_{v}\circ
R_{v}^{(3)})& =0. \\
&
\end{split}
\label{eq:derL5}
\end{equation}

In fact, if ${\operatorname{tr}}(R_{v}\circ R_{v}^{\prime })=0$ for all $
m\in M$ and all unit vector $v\in T_{m}M$, by the usual argument, it follows
that
\begin{equation*}
{\operatorname{tr}}\left( R_{\gamma _{v}^{\prime }(t)}\circ R_{\gamma
_{v}^{\prime }(t)}^{\prime }\right) =0\quad \text{ along }\quad \gamma
_{v}(t),
\end{equation*}
since $\left| \gamma _{v}^{\prime }(t)\right| =1.$ Then, from Proposition
\ref{pro:Der2} we get \eqref{eq:derL5t} deriving twice the preceding
equality. The equivalence is immediate, applying the above argument to the
equalities given by \eqref{eq:derL5}.

Now, we continue proving \eqref{eq:L7}. From \eqref{eq:opCparticular} we see
that $L_{3}=0$ gives
\begin{eqnarray*}
L_{7} &=&-\tfrac{7}{12}{\operatorname{tr}}(9R_{v}^{(5)}+20R_{v}\circ
R_{v}^{(3)}+54R_{v}^{\prime }\circ R_{v}^{\prime \prime }+32R_{v}^{\prime
}\circ R_{v}^{2}) \\
&=&-\tfrac{7}{6}{\operatorname{tr}}(10R_{v}\circ R_{v}^{(3)}+27R_{v}^{\prime
}\circ R_{v}^{\prime \prime }+16R_{v}^{\prime }\circ R_{v}^{2})
\end{eqnarray*}
by (i) of Proposition \ref{pro:L5}.

If $L_{5}=0$, it follows from Proposition \ref{pro:L5} and \eqref{eq:derL5}
that $L_{7}$ is reduced to
\begin{equation*}
\begin{split}
L_{7}=& -\tfrac{7}{6}{\operatorname{tr}}(-30R_{v}^{\prime \prime }\circ
R_{v}^{\prime }+27R_{v}^{\prime }\circ R_{v}^{\prime \prime
}+16R_{v}^{\prime }\circ R_{v}^{2}) \\
=& -\tfrac{7}{6}{\operatorname{tr}}(-3R_{v}^{\prime }\circ R_{v}^{\prime
\prime }+16R_{v}^{\prime }\circ R_{v}^{2}).
\end{split}
\end{equation*}
Thus, the condition $L_{7}=0$ and equality \eqref{eq:L7} are equivalent.

Finally, due to \cite[Lemma 2.3]{D10} and Remark \ref{rem:Der1}, from
\eqref{eq:L7} we get
\begin{equation*}
\begin{split}
& \frac{d}{dt}\left( {\operatorname{tr}}\left( 32R_{\gamma _{v}^{\prime
}(t)}^{3}-9R_{\gamma _{v}^{\prime }(t)}^{\prime }\circ R_{\gamma
_{v}^{\prime }(t)}^{\prime }\right) \right) = \\
& =6{\operatorname{tr}}\left( 16R_{\gamma _{v}^{\prime }(t)}^{\prime }\circ
R_{\gamma _{v}^{\prime }(t)}^{2}-3R_{\gamma _{v}^{\prime }(t)}^{\prime
}\circ R_{\gamma _{v}^{\prime }(t)}^{\prime \prime }\right) =0
\end{split}
\end{equation*}
since $\left| \gamma _{v}^{\prime }(t)\right| =1.$ Thus,
\begin{equation*}
\hspace{1cm}{\operatorname{tr}}\left( 32R_{\gamma _{v}^{\prime
}(t)}^{3}-9R_{\gamma _{v}^{\prime }(t)}^{\prime }\circ R_{\gamma
_{v}^{\prime }(t)}^{\prime }\right) ={\operatorname{tr}}\left(
32R_{v}^{3}-9R_{v}^{\prime }\circ R_{v}^{\prime }\right)
\end{equation*}
along $\gamma _{v}(t)$, which means that ${\operatorname{tr}}\left(
32R_{v}^{3}-9R_{v}^{\prime }\circ R_{v}^{\prime }\right) $ is invariant
under the geodesic flow. Thus, the equivalence in the statement of the
proposition is shown. \end{proof}

\section{Geometric properties of D'Atri spaces of type $k$}

\label{sec:k-DAtri}In this section, we will prove a new geometric property
of D'Atri spaces that are also $k$-D'Atri for some $k=3,\dots ,n-1$, related
to Jacobi operators along geodesics which continues the results of
Proposition \ref{pro:2.2D10} below, where it is proved that ${
\operatorname{tr}}(R_{v})$, ${\operatorname{tr}}(R_{v}^{2})$ are invariant
under the geodesic flow in any $k$-D'Atri space for some $k=1,\dots ,n-1$.
The following proposition is the key to prove our main result, Theorem \ref
{th:k-D'A property}.

\begin{proposition}
\label{pro:2.2D10} If $M$ is $k$-D'Atri for
some $k\geq 1$, then

\begin{itemize}
\item[(i)]  ${\operatorname{tr}}R_{v}^{(k)}=0$ for all $k\geq 1$ and all
unit $v\in T_{m}M$.

\item[(ii)]  Especially, ${\operatorname{tr}}R_{v}$ is invariant under the
geodesic flow (i.e. condition $L_{3}=0$ is satisfied).

\item[(iii)]  ${\operatorname{tr}}(R_{v}\circ R_{v}^{\prime })=0$ for all $
v\in T_{m}M$, $|v|=1$ or equivalently, ${\operatorname{tr}}R_{v}^{2}$ is
invariant under the geodesic flow (i.e. condition $L_{5}=0$ is satisfied).
\end{itemize}

In particular, $M$ has constant scalar curvature.
\end{proposition}

\begin{proof}

The statements (i), (ii) hold according to \cite[Proposition 2.2]{D10}. The part (iii) in the same Proposition was proved incorrectly. The correct argument will be given in Proposition \ref{pro:correction} below. We first need some auxiliary calculations.

In that follows we will use Newton's relations (see \cite[A.IV.70]{B}):
Given $n$ real numbers, $\lambda _{1},...,\lambda _{n},$ and any natural $
k=1,...,n,$ if we denote by $s_{k}=\sum_{i=1}^{n}\lambda _{i}^{k}$ and by $
\sigma _{k}$ their associated $k$-th elementary symmetric functions, then
\begin{equation}
s_{k}-s_{k-1}\sigma _{1}+s_{k-2}\sigma _{2}+...+(-1)^{k-1}s_{1}\sigma
_{k-1}+(-1)^{k}k\sigma _{k}=0,\,\,\,k\leq n.  \label{eq:Newton}
\end{equation}

Let $v\in T_{m}M$ be a fix unit vector and let $t>0$ be a fix small real
number. Recall from Section \ref{sec:Ledger} that for each natural number $
l\geq 1$ we denote by $O(v,t^{l})=\sum_{i=l}^{\infty }\alpha _{i}t^{i},$
obtained from the Taylor expansion of $C_{v}(t)=tS_{v}(t)=\sum_{j=0}^{\infty
}\alpha _{j}t^{j}$, with $\alpha _{j}=\alpha _{j}(v)=\frac{1}{j!}
C_{v}^{(j)}(0).$ We expand $C_{v}(t)^{l}$ for each $l\geq 1$, as follows:
\begin{equation*}
\begin{split}
C_{v}(t)^{l}& =\left( I+\alpha _{2}t^{2}+\alpha _{3}t^{3}+\alpha
_{4}t^{4}+\alpha _{5}t^{5}+\alpha _{6}t^{6}+\alpha
_{7}t^{7}+O(v,t^{8})\right) ^{l} \\
& =I+\binom{l}{1}\{\alpha _{2}t^{2}+\alpha _{3}t^{3}+\alpha _{4}t^{4}+\alpha
_{5}t^{5}+\alpha _{6}t^{6}+\alpha _{7}t^{7}+O(v,t^{8})\} \\
& \hspace{0.7cm}+\binom{l}{2}\{\alpha _{2}t^{2}+\alpha _{3}t^{3}+\alpha
_{4}t^{4}+\alpha _{5}t^{5}+\alpha _{6}t^{6}+\alpha _{7}t^{7}+O(v,t^{8})\}^{2}
\\
& \hspace{0.7cm}+\binom{l}{3}\{\alpha _{2}t^{2}+\alpha _{3}t^{3}+\alpha
_{4}t^{4}+\alpha _{5}t^{5}+\alpha _{6}t^{6}+\alpha _{7}t^{7}+O(v,t^{8})\}^{3}
\\
& \hspace{0.7cm}+\cdots
\end{split}
\end{equation*}
That is,
\begin{equation*}
\begin{split}
C_{v}(t)^{l}=\,I& +t^{2}\binom{l}{1}\alpha _{2}+t^{3}\binom{l}{1}\alpha
_{3}+t^{4}\left\{ \binom{l}{1}\alpha _{4}+\binom{l}{2}\alpha _{2}^{2}\right\}
\\
& +t^{5}\left\{ \binom{l}{1}\alpha _{5}+\binom{l}{2}(\alpha _{2}\alpha
_{3}+\alpha _{3}\alpha _{2})\right\} \\
& +t^{6}\left\{ \binom{l}{1}\alpha _{6}+\binom{l}{2}(\alpha _{2}\alpha
_{4}+\alpha _{4}\alpha _{2}+\alpha _{3}^{2})+\binom{l}{3}\alpha
_{2}^{3}\right\} \\
& +t^{7}\left\{ \binom{l}{1}\alpha _{7}+\binom{l}{2}(\alpha _{2}\alpha
_{5}+\alpha _{5}\alpha _{2}+\alpha _{3}\alpha _{4}+\alpha _{4}\alpha
_{3})\right. \\
& \left. \hspace{1cm}+\binom{l}{3}(\alpha _{2}^{2}\alpha _{3}+\alpha
_{2}\alpha _{3}\alpha _{2}+\alpha _{3}\alpha _{2}^{2})\right\} +O(v,t^{8}) \\
&
\end{split}
\end{equation*}
Hence, setting $s_{l}=s_{l}(v)$ and $\gamma _{j}=\gamma _{j}(v)$, we have

\begin{equation}
\begin{split}
s_{l}={\operatorname{tr}}C_{v}(t)^{l}=\,n-1& +t^{2}\binom{l}{1}\gamma
_{1}+t^{3}\binom{l}{1}\gamma _{2}+t^{4}\left\{ \binom{l}{1}\gamma _{3}+
\binom{l}{2}\gamma _{4}\right\} \\
& +t^{5}\left\{ \binom{l}{1}\gamma _{5}+\binom{l}{2}\gamma _{6}\right\} \\
& +t^{6}\left\{ \binom{l}{1}\gamma _{7}+\binom{l}{2}\gamma _{8}+\binom{l}{3}
\gamma _{9}\right\} \\
& +t^{7}\left\{ \binom{l}{1}\gamma _{10}+\binom{l}{2}\gamma _{11}+\binom{l}{3
}\gamma _{12}\right\} +O(v,t^{8})
\end{split}
\label{eq:sl}
\end{equation}
where
\begin{equation}
\begin{gathered} \gamma_{1}={\operatorname{tr}}\alpha_2, \quad
\gamma_{2}={\operatorname{tr}}\alpha_3, \quad
\gamma_{3}={\operatorname{tr}}\alpha_4, \quad
\gamma_{4}={\operatorname{tr}}(\alpha_2^2), \quad
\gamma_{5}={\operatorname{tr}}\alpha_5, \\
\gamma_{6}=2{\operatorname{tr}}(\alpha_2\alpha_3), \quad
\gamma_{7}={\operatorname{tr}}\alpha_6, \quad \gamma_{8}=2
{\operatorname{tr}}(\alpha_2\alpha_4)+{\operatorname{tr}}(\alpha_3^2), \quad
\gamma_{9}={\operatorname{tr}}(\alpha_2^3),\\
\gamma_{10}={\operatorname{tr}}\alpha_7, \quad
\gamma_{11}=2{\operatorname{tr}}(\alpha_2\alpha_5)+2{\operatorname{tr}}(
\alpha_3\alpha_4),\quad
\gamma_{12}=3{\operatorname{tr}}(\alpha_2^2\alpha_3).\\ \end{gathered}
\label{eq:gammas}
\end{equation}

\begin{proposition}
\label{pro:correction} If $M$ is a $k$-D'Atri space for some $k\geq 1,$ then
\begin{equation*}
{\operatorname{tr}}\left( R_{v}\circ R_{v}^{\prime }\right) =0
\end{equation*}
for all unit vectors $v\in T_{m}M.$
\end{proposition}

\begin{proof} \ Using $(i)$ and $(ii)$ of Proposition \ref{pro:2.2D10},
\eqref{eq:opCparticular} and the definitions of $\gamma _{j}$, $j=1,...,6$,
given in \eqref{eq:gammas} we get,
\begin{equation*}
\gamma _{2}={\operatorname{tr}}\alpha _{3}=-\tfrac{1}{4}{\operatorname{tr}}
R_{v}^{\prime }=0,
\end{equation*}
\begin{equation*}
\gamma _{5}={\operatorname{tr}}\alpha _{5}=-\tfrac{10}{13\cdot 5!}\left( {
\operatorname{tr}}(R_{v}^{(3)})+{\operatorname{tr}}(R_{v}\circ R_{v}^{\prime
})\right) =-\tfrac{1}{156}{\operatorname{tr}}(R_{v}\circ R_{v}^{\prime }),
\end{equation*}
\begin{equation*}
\gamma _{6}={\operatorname{tr}}(2\alpha _{2}\alpha _{3})=\tfrac{1}{6}{
\operatorname{tr}}(R_{v}\circ R_{v}^{\prime })=-26\gamma _{5}.
\end{equation*}

Therefore, for all $l\geq 1$, \eqref{eq:sl} can be written as
\begin{equation}
\begin{split}
s_{l}=& \,n-1+t^{2}\binom{l}{1}\gamma _{1}+t^{4}\left\{ \binom{l}{1}\gamma
_{3}+\binom{l}{2}\gamma _{4}\right\} \\
& +t^{5}\left\{ -\frac{1}{156}\left( \binom{l}{1}-26\binom{l}{2}\right) {
\operatorname{tr}}(R_{v}\circ R_{v}^{\prime })\right\} +O(v,t^{6}).
\end{split}
\label{eq:slt6}
\end{equation}

Now, by Newton's formula \eqref{eq:Newton} we get
\begin{equation*}
(-1)^{l}\sigma _{l}(v,t)=-s_{l}+s_{l-1}\sigma _{1}(v,t)+\cdots
+(-1)^{l-1}s_{1}\sigma _{l-1}(v,t),\text{ }l=1,...,k
\end{equation*}
and applying \eqref{eq:slt6} in the equality above for each $l=1,...,k$, we
obtain
\begin{equation*}
\begin{split}
\sigma _{k}(v,t)=& \binom{n-1}{k}+t^{2}\binom{n-2}{k-1}\gamma
_{1}+t^{4}\left\{ \binom{n-2}{k-1}\gamma _{3}+\frac{1}{2}\binom{n-3}{k-2}
\left( \gamma _{1}^{2}-\gamma _{4}\right) \right\} \\
& +t^{5}\left\{ -\frac{1}{156}\left( \binom{n-2}{k-1}+13\binom{n-3}{k-2}
\right) {\operatorname{tr}}(R_{v}\circ R_{v}^{\prime })\right\} +O(v,t^{6}).
\end{split}
\end{equation*}
From \eqref{eq:gammas} and the facts ${\operatorname{tr}}\alpha _{2}(v)={
\operatorname{tr}}\alpha _{2}(-v)$ $({\operatorname{tr}}R_{v}={
\operatorname{tr}}R_{-v})$ and ${\operatorname{tr}}\alpha _{4}(v)={
\operatorname{tr}}\alpha _{4}(-v)$ $({\operatorname{tr}}R_{v}^{2}={
\operatorname{tr}}R_{-v}^{2},$ ${\operatorname{tr}}R_{v}^{\prime \prime }=0)$
, we see that $\gamma _{i}(v)=\gamma _{i}(-v)$, $i=1,3,4$. Thus, under the
assumption that $M$ is a $k$-D'Atri space for some $k\geq 3$ and setting $
O(\pm v,t)=O(v,t)-O(-v,t),$ we have
\begin{equation*}
\begin{split}
0=& \ \sigma _{k}(v,t)-\sigma _{k}(-v,t) \\
=& \ t^{5}\left\{ -\frac{1}{156}\left( \binom{n-2}{k-1}+13\binom{n-3}{k-2}
\right) {\operatorname{tr}}(R_{v}\circ R_{v}^{\prime })\right\} +O(\pm
v,t^{6}).
\end{split}
\end{equation*}
This gives
\begin{equation*}
-\frac{1}{156}\left( \binom{n-2}{k-1}+13\binom{n-3}{k-2}\right) {
\operatorname{tr}}(R_{v}\circ R_{v}^{\prime })+O(\pm v,t)=0
\end{equation*}
for any small $t>0$. Finally, we conclude de desired result taking the limit
as $t\rightarrow 0$.
\end{proof}

The proof of Proposition \ref{pro:2.2D10} is herewith completed.
\end{proof}

We note that by \cite{B-V.92} D'Atri spaces in dimension $3$ are homogeneous
and have the property that the eigenvalues of the Jacobi operator are
constant along each geodesic (they are ${\frak{C}}$-spaces). Thus, ${
\operatorname{tr}}(R_{v}^{3})$ is invariant under the geodesic flow and the
next theorem is also valid for $n=3$.

\begin{theorem}
\label{th:k-D'A property} If $M$ is a $n$-dimensional D'Atri space with $
n\geq 4$ which is also a D'Atri space of type $k$ for some $k=3,\dots ,n-1$,
then
\begin{equation}
{\operatorname{tr}}(R_{v}^{\prime }\circ R_{v}^{2})=0\text{ for all unit
vector }v\in T_{m}M.  \label{eq:C1}
\end{equation}
Equivalently, ${\operatorname{tr}}(R_{v}^{3})$ is invariant under the
geodesic flow.
\end{theorem}

\begin{proof} \ Under our hypothesis, we know that all odd Ledger
conditions are satisfied due to Remark \ref{rem:D'A iff Ledger}. Thus, we
directly get from \eqref{eq:gammas} using Definition \ref{def:Ledger} that $
\gamma _{2}={\operatorname{tr}}\alpha _{3}=-\tfrac{1}{4}{\operatorname{tr}}
R_{v}^{\prime }=0$ and analogously, $\gamma _{5}=0=\gamma _{10}$ ($L_{5}=0=L_{7}$). Moreover, applying Proposition \ref{pro:L5},
\eqref{eq:derL5} and Proposition \ref{pro:L7}, we also have from
\eqref{eq:opCparticular} that
\begin{equation*}
\begin{split}
\gamma _{6}=& {\operatorname{tr}}(2\alpha _{2}\alpha _{3})=\tfrac{1}{6}{
\operatorname{tr}}(R_{v}\circ R_{v}^{\prime })=0, \\
\gamma _{11}=& \,2{\operatorname{tr}}(\alpha _{2}\alpha _{5})+2{
\operatorname{tr}}(\alpha _{3}\alpha _{4})=\tfrac{1}{5!}{\operatorname{tr}}
(C_{v}^{(2)}(0)C_{v}^{(5)}(0))+\tfrac{1}{3\cdot 4!}{\operatorname{tr}}
(C_{v}^{(3)}(0)C_{v}^{(4)}(0)) \\
=& \tfrac{1}{54}{\operatorname{tr}}(R_{v}\circ R_{v}^{(3)})+\left( \tfrac{1}{
54}+\tfrac{1}{90}\right) {\operatorname{tr}}(R_{v}^{\prime }\circ R_{v}^{2})+
\tfrac{1}{20}{\operatorname{tr}}(R_{v}^{\prime }\circ R_{v}^{\prime \prime })
\\
=& \left( \tfrac{1}{54}+\tfrac{1}{90}\right) {\operatorname{tr}}
(R_{v}^{\prime }\circ R_{v}^{2})+\left( \tfrac{1}{20}-\tfrac{3}{54}\right) {
\operatorname{tr}}(R_{v}^{\prime }\circ R_{v}^{\prime \prime }) \\
=& \tfrac{1}{540}\left( 16{\operatorname{tr}}(R_{v}^{\prime }\circ
R_{v}^{2})-3{\operatorname{tr}}(R_{v}^{\prime }\circ R_{v}^{\prime \prime
})\right) =0.
\end{split}
\end{equation*}
Therefore, for each $l\geq 1$, \eqref{eq:sl} is reduced to
\begin{equation}
\begin{split}
s_{l}={\operatorname{tr}}C_{v}(t)^{l}=\,n-1& +t^{2}\binom{l}{1}\gamma
_{1}+t^{4}\left\{ \binom{l}{1}\gamma _{3}+\binom{l}{2}\gamma _{4}\right\} \\
& +t^{6}\left\{ \binom{l}{1}\gamma _{7}+\binom{l}{2}\gamma _{8}+\binom{l}{3}
\gamma _{9}\right\} \\
& +t^{7}\left\{ \binom{l}{3}\gamma _{12}\right\} +O(v,t^{8}).
\end{split}
\label{eq:sk}
\end{equation}

Now, denoting $\sigma _{l}=\sigma _{l}(v,t)$ and substituting \eqref{eq:sk}
in the recursive formula \eqref{eq:Newton} for each $l=1,...,k,$ we obtain

\begin{equation*}
\begin{split}
\sigma _{k}& =\,\binom{n-1}{k}+t^{2}\binom{n-2}{k-1}\gamma _{1}+t^{4}\left\{
\binom{n-2}{k-1}\gamma _{3}+\frac{1}{2}\binom{n-3}{k-2}\left( \gamma
_{1}^{2}-\gamma _{4}\right) \right\} \\
+& t^{6}\left\{ \binom{n-2}{k-1}\gamma _{7}+\binom{n-3}{k-2}\left( \gamma
_{1}\gamma _{3}-\frac{\gamma _{8}}{2}\right) +\binom{n-4}{k-3}\left( \frac{
\gamma _{1}^{3}}{6}-\frac{\gamma _{1}\gamma _{4}}{2}+\frac{\gamma _{9}}{3}
\right) \right\} \\
+& t^{7}\left\{ \frac{1}{3}\binom{n-4}{k-3}\gamma _{12}\right\} +O(v,t^{8}).
\end{split}
\end{equation*}
Moreover, using \eqref{eq:gammas} and the facts ${\operatorname{tr}}\alpha
_{2}(v)={\operatorname{tr}}\alpha _{2}(-v)$, ${\operatorname{tr}}\alpha
_{4}(v)={\operatorname{tr}}\alpha _{4}(-v)$, ${\operatorname{tr}}\alpha
_{6}(v)={\operatorname{tr}}\alpha _{6}(-v)$ $({\operatorname{tr}}
(R_{v}^{\prime }\circ R_{v}^{\prime })={\operatorname{tr}}(R_{-v}^{\prime
}\circ R_{-v}^{\prime }),{\operatorname{tr}}(R_{v}\circ R_{v}^{\prime \prime
})={\operatorname{tr}}(R_{-v}\circ R_{-v}^{\prime \prime }),{
\operatorname{tr}}R_{v}^{3}={\operatorname{tr}}R_{-v}^{3},{\operatorname{tr}}
R_{v}^{(4)}=0)$ and ${\operatorname{tr}}\alpha _{3}(v)=-{\operatorname{tr}}
\alpha _{3}(-v)$ $({\operatorname{tr}}R_{v}^{\prime }=-{\operatorname{tr}}
R_{-v}^{\prime }),$ it is easy to realize that
\begin{equation*}
\gamma _{i}(v)=\gamma _{i}(-v),\text{ }i=1,3,4,7,8,9,\text{ and }\gamma
_{12}(v)=-\gamma _{12}(-v).
\end{equation*}
Thus, under the assumption that $M$ is a $k$-D'Atri space for some $k\geq 3$
and setting $O(\pm v,t)=O(v,t)-O(-v,t),$ we obtain
\begin{equation*}
0=\sigma _{k}(v,t)-\sigma _{k}(-v,t)=t^{7}\left\{ \frac{2}{3}\binom{n-4}{k-3}
\gamma _{12}\right\} +O(\pm v,t^{8}).
\end{equation*}
This gives
\begin{equation*}
\frac{2}{3}\binom{n-4}{k-3}\gamma _{12}+O(\pm v,t)=0\text{ for any small }
t>0,
\end{equation*}
which implies that $\gamma _{12}=0$ for $n\geq 4$ and $k\geq 3,$ taking into
account that $\lim_{t\rightarrow 0}O(\pm v,t)=0.$ Therefore, proceeding as
before,
\begin{equation*}
0=\gamma _{12}=3{\operatorname{tr}}\left( \alpha _{2}^{2}\alpha _{3}\right)
=-\tfrac{1}{12}{\operatorname{tr}}\left( R_{v}^{2}\circ R_{v}^{\prime
}\right)
\end{equation*}
and we get the desired condition \eqref{eq:C1}.

Finally, due to \cite[Lemma 2.3]{D10} and \eqref{eq:C1} we have that
\begin{equation*}
\frac{d}{dt}\left( {\operatorname{tr}}R_{\gamma _{v}^{\prime}(t)}^{3}\right) =3{\operatorname{tr}}\left( R_{\gamma _{v}^{\prime}(t)}^{2}\circ R_{\gamma _{v}^{\prime }(t)}^{\prime }\right) =0,
\end{equation*}
since $\left| \gamma _{v}^{\prime }(t)\right| =1.$ Then, ${\operatorname{tr}}
\left( R_{\gamma _{v}^{\prime }(t)}^{3}\right) ={\operatorname{tr}(}
R_{v}^{3})$ along $\gamma _{v}(t)$ which means that ${\operatorname{tr}(}
R_{v}^{3})$ is invariant under the geodesic flow. \end{proof}

Note that from Proposition \ref{pro:2.2D10}, with the same proof above and
applying Proposition \ref{pro:L7}, we also have the following consequence.

\begin{corollary}
If $M$ is a $k$-D'Atri space for some $k=3,...,n-1$ that satisfies the third
odd Ledger condition $L_{7}=0$ (or, equivalently, the trace ${\operatorname{tr}}(32R_{v}^{3}-9R_{v}^{\prime }\circ R_{v}^{\prime })$ is invariant under the geodesic flow), then ${\operatorname{tr}}(R_{v}^{3})$ is invariant under the geodesic flow.
\end{corollary}

Non-symmetric Damek-Ricci spaces are D'Atri spaces which are not 3-D'Atri
(see \cite[Theorem 3.2, (ii)]{D10}). In the next section we will use the
previous theorem to prove that non-symmetric Damek-Ricci spaces are not $k$
-D'Atri for any $k\geq 3$ (see Corollary \ref{cor:kDAtriSym}). Furthermore,
no examples are known of $k$-D'Atri spaces for some $k\geq 3$ which are not
D'Atri. Therefore, it still remains open the question if a $k$-D'Atri space
for some $k\geq 3$ is a D'Atri space.

\section{Applications to D'Atri spaces of Iwasawa type} \label{sec:kDatri and Iwasawa}

We recall that a solvable Lie algebra ${\frak{s}}$ with inner product $
\left\langle ,\right\rangle $ is a metric Lie algebra of Iwasawa type, if it
satisfies the conditions

\begin{itemize}
\item[(i)]  $\,{\frak{s}}={\frak{n}}\oplus {\frak{a}}$ where ${\frak{n}}=[{
\frak{s}},{\frak{s}}]$ and ${\frak{a}},$ the orthogonal complement of ${
\frak{n}}$, is abelian.

\item[(ii)]  The operators $\left. {\operatorname{ad}}_{H}\right| _{{\frak{n}
}}$ are symmetric and non zero for all $H\in {\frak{a}}$.

\item[(iii)]  There exits $H_{0}\in {\frak{a}}\,$\ such that ${
\operatorname{ad}}_{H_{0}}|_{{\frak{n}}}\,$ has all positive eigenvalues.
\end{itemize}

The simply connected Lie group $S$ with Lie algebra ${\frak{s}}$ and left
invariant metric $g$ induced by the inner product $\left\langle
,\right\rangle $ will be called a space of Iwasawa type. The algebraic rank
of $S$ (equivalently ${\frak{s}}$) is defined by $\dim {\frak{a}}$.

In that follows we assume that $M=S$ and fix $m=e,\,\,$the identity of the
group $S;$ we identify ${\frak{s}}$ with $T_{e}S$ by $X=\widetilde{X}_{e},$ $
\,$where $\widetilde{X}$ denotes the left invariant field on $S$ associated
to $X\in {\frak{s}}.$ The Levi Civita connection $\widetilde{\nabla }\,$ at $
e,$ denoted by $\nabla ,\,\,$and the curvature tensor $R$ associated to the
metric can be computed by
\begin{eqnarray*}
2\left\langle \nabla _{X}Y,Z\right\rangle &=&\left\langle
[X,Y],Z\right\rangle -\left\langle [Y,Z],X\right\rangle +\left\langle
[Z,X],Y\right\rangle \\
R(X,Y) &=&[\nabla _{X},\nabla _{Y}]-\nabla _{[X,Y]}
\end{eqnarray*}
for any $X,Y,Z$ in ${\frak{s}}$.
Using the above formula one obtains for each $X,Y \in {\frak{s}}$ and $H\in \frak{a}$ that
\begin{eqnarray*}
2\left\langle \nabla _{H}X,Y\right\rangle  &=&\left\langle
[H,X],Y\right\rangle -\left\langle [X,Y],H\right\rangle +\left\langle
[Y,H],X\right\rangle  \\
&=&\left\langle [H,X],Y\right\rangle -\left\langle [H,Y],X\right\rangle
=\left\langle \text{ad}_{H}X,Y\right\rangle -\left\langle \text{ad}%
_{H}Y,X\right\rangle =0
\end{eqnarray*}
since $\text{ad}_H$ is symmetric for all $H\in \frak{a}$. Then, $\nabla _{H}=0$ and hence $R_{H}=-{\operatorname{ad}}
_{H}^{2}$. Moreover, $R_{H}^{\prime }=0$ for all $H\in {\frak{a}}$ since $\gamma _{H}(t)=\exp tH$
($\nabla _{H}H=0$), $\gamma _{H}^{\prime }(t)=\widetilde{H}_{\exp tH}=(\text{d} L_{\exp tH})_{e}H$ ($\exp$ denotes the exponential map of the Lie group $S$) and the metric is left invariant, we have that
\begin{equation*}
\text{\thinspace }R_{H}^{\prime }=\left. \left( \nabla _{\gamma _{H}^{\prime
}(t)}\text{\thinspace }R_{\gamma _{H}^{\prime }(t)}\right) \right|
_{t=0}=\left. (dL_{\exp tH})_{e}(\nabla _{H}R_{H})\right| _{t=0}=\nabla
_{H}R_{H}=0.
\end{equation*}
Note that by definition, $\nabla _{H}R_{H}(X)=\nabla
_{H}(R_{H}X)-R_{H}(\nabla _{H}X)=0$ for any $X\in \frak{s}$.

In addition, we recall the following known fact proved in \cite[Proposition 2.1]{D09}: For each unit vector $X$ in $\frak{s}=T_{e}S$, we denote by $\gamma _{X}(t)$ the geodesic in $S$ with $\gamma _{X}(0)=e$, and by $x(t)$ the curve uniquely defined in the unit sphere of $\frak{s}$ ($\left| x(t)\right| =1$) obtained by the isometry $\left( \text{dL}_{\gamma
_{X}(t)}\right) _{e}:\frak{s}=T_{e}S\rightarrow T_{\gamma _{X}(t)}S$. We express
$\gamma _{X}^{\prime }(t)=\left( \text{dL}_{\gamma _{X}(t)}\right) _{e}x(t) \in T_{\gamma _{X}(t)}S \text{ for all real }t.$
Then, either
\begin{equation}\label{eq:IwaGePro}
x(t)=H\in \frak{a}\text{ and }\gamma _{H}(t)=\exp tH \text{ or }
\lim_{n\rightarrow \infty }x(t_{n})=H
\end{equation}
for some $H\in \frak{a}$ and some sequence of real numbers $\{t_{n}\}_{n\in N}$.

It is worth pointing out that irreducible homogeneous and simply connected
D'Atri spaces of nonpositive curvature can be represented as Iwasawa type
spaces (they are Einstein) and they are either symmetric spaces of higher
rank or Damek-Ricci spaces in the case of rank one, including the rank one
symmetric spaces of noncompact type (see \cite{H95},\thinspace \cite{H98}, \cite{H06}).
Moreover, D'Atri spaces (without curvature restrictions) of Iwasawa type and
algebraic rank one are Damek-Ricci spaces, since they are harmonic (see \cite
{D09}). For the definition of Damek-Ricci spaces and its properties see \cite
{B-Tr-V}. Now, we analyze ${\frak{C}}$-spaces and the D'Atri condition on
Iwasawa type spaces.

\begin{proposition}
\label{pro:IwaCpro} If $S$ is a ${\frak{C}}$-space of Iwasawa type, then $S$
has nonpositive sectional curvature.
\end{proposition}

\begin{proof} \ Let $X\in {\frak{s}}$ be a unit vector. Since $S$ is a ${
\frak{C}}$-space, the eigenvalues of $R_{\gamma _{X}^{\prime }(t)}$ are
constant along $\gamma _{X}(t)$ and the characteristic polynomial of $
R_{\gamma _{X}^{\prime }(t)}$ can be write as
\begin{equation}
\det \left( \lambda {\operatorname{Id}}-R_{\gamma _{X}^{\prime }(t)}\right)
=\det \left( \lambda {\operatorname{Id}}-R_{X}\right) .  \label{eq:Cpro}
\end{equation}
Moreover, using the previous notation
\begin{eqnarray*}
\det \left( \lambda {\operatorname{Id}}-R_{\gamma _{X}^{\prime }(t)}\right)
&=&\det \left( (\text{d}L_{\gamma _{X}(t)})_{e}\circ (\lambda {
\operatorname{Id}}-R_{x(t)}\text{)}\circ (\text{d}L_{\gamma
_{X}(t)})_{e}^{-1}\right) \\
&=&\det \left( \lambda {\operatorname{Id}}-R_{x(t)}\right) \text{ for all }\;t\in {\mathbb R}.
\end{eqnarray*}
Then, taking lim as $t\rightarrow \infty $ in the above equality, it follows
from \eqref{eq:IwaGePro} and \eqref{eq:Cpro} that
\begin{equation*}
\det \left( \lambda {\operatorname{Id}}-R_{X}\right) =\det \left( \lambda {
\operatorname{Id}}-R_{H}\right) =\det \left( \lambda {\operatorname{Id}}+{
\operatorname{ad}}_{H}^{2}\right)
\end{equation*}
for some $H\in {\frak{a}}$. Hence, the eigenvalues of $R_{X}$ are exactly
those of $-$ad$_{H}^{2}$ and consequently, the sectional curvature of $S$ is
nonpositive. \end{proof}

\begin{corollary}
\label{cor:IwaCproDAtri} Let $S$ be an irreducible space of Iwasawa type.
Then, $S$ is a symmetric space of noncompact type if and only if it is a
D'Atri space and a ${\frak{C}}$-space.
\end{corollary}

\begin{proof} \ It is well-known that symmetric spaces are an important
subclass of D'Atri spaces and ${\frak{C}}$-spaces.

On the other hand, under our hypothesis, by Proposition \ref{pro:IwaCpro} $S$
is an irreducible D'Atri space of nonpositive curvature. It follows from
\cite[Theorem 4.7]{H95} that either $S$ is symmetric of noncompact type of
higher rank, or $S$ is a D'Atri (harmonic) space of algebraic rank one. In
the last case, by applying \cite[Corollary 2.2.]{D06} $S$ is a symmetric
space of rank one, since
\begin{equation*}
\text{tr}R_{X}^{k}=\text{tr}\left( -\text{ad}_{H_{0}}^{2}\right) ^{k}\text{
for all }X\in \frak{s},\;\left| X\right| =1\text{ and }k=1,...,n-1,
\end{equation*}
which means that $S$ is $k$-stein for all $k=1,...,n-1.$ \end{proof}

In the next result we stablish a number of curvature conditions needed to
determine whether a space of Iwasawa type is symmetric.

\begin{theorem}
\label{th:Ledger+Iwasawa} Let $S$ be a space of Iwasawa type. Then,

\begin{itemize}
\item[(i)]  $S$ is a symmetric space if and only if ${\operatorname{tr}}
(32R_{X}^{3}-9R_{X}^{\prime }\circ R_{X}^{\prime })$ and ${\operatorname{tr}}
(R_{X}^{3})$ are invariant under the geodesic flow for all $X\in \frak{s}$
with $|X|=1$.

\item[(ii)]  $S$ is a symmetric space if and only if the three first odd
Ledger conditions are satisfied and ${\operatorname{tr}}(R_{X}^{\prime
}\circ R_{X}^{2})=0$ for all $X\in \frak{s}$ with $|X|=1$.
\end{itemize}
\end{theorem}

\begin{proof}\ Let $X\in {\frak{s}}$ be a unit vector.

\begin{itemize}
\item[(i)]  If ${\operatorname{tr}}(32R_{X}^{3}-9R_{X}^{\prime }\circ
R_{X}^{\prime })$ and ${\operatorname{tr}}(R_{X}^{3})$ are invariant under
the geodesic flow, then ${\operatorname{tr}}(R_{X}^{\prime }\circ
R_{X}^{\prime })$ is also invariant under the geodesic flow. That is,
\begin{equation*}
{\operatorname{tr}}(R_{\gamma _{X}^{\prime }(t)}^{\prime }\circ R_{\gamma
_{X}^{\prime }(t)}^{\prime })={\operatorname{tr}}(R_{X}^{\prime }\circ
R_{X}^{\prime })\quad \text{ for all }\;t\in {\mathbb R}.
\end{equation*}
Finally, taking lim as $t\rightarrow \infty $, from \eqref{eq:IwaGePro} and
the fact $R_{H}^{\prime }=0$ $(H\in \frak{a})$ we have that
\begin{equation*}
0={\operatorname{tr}}(R_{H}^{\prime }\circ R_{H}^{\prime })={
\operatorname{tr}}(R_{X}^{\prime }\circ R_{X}^{\prime }),\quad X\in {\frak{s}
},\,|X|=1.
\end{equation*}
Hence, $R_{X}^{\prime }=0$ because $\operatorname{tr}(R'_X\circ R'_X)=\sum_{i=1}^n |R'_X Y_i|^2=0$ implies $R'_X Y_i=0$ for all $i=1,\dots, n$ where $\{Y_i\}$ is a basis of $\frak{s}$. Consequently, $S$ is symmetric since $\nabla
R=0$ (see for example \cite{OS84} or \cite[p. 59]{B-V.92}).

\item[(ii)]  If $L_{3}=0$, $L_{5}=0$, $L_{7}=0$ and ${\operatorname{tr}}
(R_{X}^{\prime }\circ R_{X}^{2})=0$ for all unit vector $X\in {\frak{s}}$,
by \eqref{eq:L7} we get ${\operatorname{tr}}(R_{X}^{\prime }\circ
R_{X}^{\prime \prime })=0$. Thus, ${\operatorname{tr}}(R_{X}^{\prime }\circ
R_{X}^{\prime })$ is invariant under the geodesic flow and consequently, $S$
is symmetric (see the proof of $(i)$).
\end{itemize}
\end{proof}

Thus, from Remark \ref{rem:D'A iff Ledger} we characterize a special
subclass of D'Atri spaces of Iwasawa type using only the three first odd
Ledger's conditions.

\begin{corollary}
Let $S$ be a space of Iwasawa type. Then, $S$ is a symmetric space if and
only if $S$ is a D'Atri space and ${\operatorname{tr}}(R_{X}^{\prime }\circ
R_{X}^{2})=0$ for all unit vector $X\in {\frak{s}}$.

Equivalently, if $S$ is a D'Atri space of Iwasawa type, then $S$ is
symmetric if and only if ${\operatorname{tr}}(R_{X}^{3})$ is invariant under
the geodesic flow.
\end{corollary}

\begin{remark}
\begin{itemize}
\item[(i)]  In particular, if $S$ has algebraic rank one, the property ${
\operatorname{tr}}(R_{X}^{3})$ being invariant under the geodesic flow is a
necessary condition in the above corollary, since nonsymmetric Damek-Ricci
spaces do not satisfy such property by \cite{D06}.

\item[(ii)]  The property ${\operatorname{tr}}(R_{X}^{3})$ being invariant
under the geodesic flow in Iwasawa type ${\frak{C}}$-spaces and the previous
corollary give an alternative proof of Corollary \ref{cor:IwaCproDAtri}.
\end{itemize}
\end{remark}

Finally, applying the main results of the preceding section and Theorem \ref
{th:Ledger+Iwasawa} we get some stronger results than the previously
obtained in \cite{D10}.

It is known that the property of $S$ being a $k$-D'Atri space for all
$k=1,...,n-1$ characterizes the symmetric spaces of noncompact type within
the class of of Iwasawa type spaces. In particular, in the class of
Damek-Ricci spaces the symmetric of noncompact type and rank one are
characterized by the $3$-D'Atri condition (see \cite[Theorem 3.2]{D10}).
Now, we generalize this result to Iwasawa type spaces, where the symmetric
ones are those which are D'Atri and $3$-D'Atri spaces.

\begin{corollary}
\label{cor:Iwa3DAtri} Let $S$ be a D'Atri space of Iwasawa type. Then, $S$
is symmetric if and only if $S$ is a $3$-D'Atri space. In such a case, $S$ is a $k$
-D'Atri space for all $k=1,...,n-1.$
\end{corollary}

\begin{proof} \ Applying the equivalence between D'Atri, $1$-D'Atri and $
2 $-D'Atri properties and \cite[Proposition 2.4]{D10}, under the assumption
that $S$ is a $3$-D'Atri space it follows that ${\operatorname{tr}}
(R_{X}^{k})$, $k=1,2,3,$ and ${\operatorname{tr}}(32R_{X}^{3}-9R_{X}^{\prime
}\circ R_{X}^{\prime })$ are invariant under the geodesic flow (see Proposition \ref{pro:L7}). The
assertion is immediate by Theorem \ref{th:Ledger+Iwasawa}. \end{proof}

Finally, using Theorem \ref{th:k-D'A property} we get an stronger consequence of Theorem \ref{th:Ledger+Iwasawa}
than the previously obtained in Corollary \ref{cor:Iwa3DAtri}.

\begin{corollary}
\label{cor:kDAtriSym} Let $S$ be a D'Atri space of Iwasawa type of dimension
$n\geq 4$. Then, $S$ is symmetric if and only if $S$ is a $k$-D'Atri space
for some $k,$ $3\leq k\leq n-1$. In such a case, $S$ is a $k$-D'Atri for all $
k=1,...,n-1.$

In particular, if $S$ is Damek-Ricci then $S$ is a $k$-D'Atri space for some
$k\geq 3$ if and only if $S$ is a rank one symmetric space of noncompact
type.
\end{corollary}

\section{Geodesic symmetries and $k$-D'Atri spaces for all $k\geq 1$}

\label{sec:GCandk-D'Atri}

In this section, we characterize $k$-D'Atri spaces for all $k\geq 1$ as the $
\frak{SC}$-spaces and we show applications on 4-dimensional homogeneous
spaces. Recall that $M$ is a \emph{$\frak{SC}$-space} (a $\frak{SP}$\emph{
-space}) if for any small real $t>0$ and any unit vector $v\in T_{m}M$ the
eigenvalues (the eigenvectors) of $S_{v}(t),$ the shape operator of the
geodesic sphere $G_{m}(t)$ centered at $m,$ are preserved by the geodesic
symmetries $s_{m}$ for all $m\in M.$ See \cite{B-P-V} for definitions and
related notions.

By definition, $s_{m}$ preserves $S_{v}(t)$ if and only if
\begin{equation*}
\left. \text{d}s_{m}\right| _{\gamma _{v}(t)}\circ S_{v}(t)=S_{-v}(t)\circ
\left. \text{d}s_{m}\right| _{\gamma _{v}(t)}.
\end{equation*}
Moreover, the following proposition is well known in the literature and we
include it here for the sake of completeness.

\begin{proposition}
\label{pro:Known}The geodesic symmetries $s_{m}$ preserve the shape
operators $S_{v}(t)$ of small geodesic $G_{m}(t)$ if and only if $S_{v}(t)$
and $\left. \text{d}s_{m}\right| _{\gamma _{v}(t)}^{-1}\circ S_{-v}(t)\circ
\left. \text{d}s_{m}\right| _{\gamma _{v}(t)}$ are simultaneously
diagonalizable and have the same eigenvalues. Equivalently, for each small
real $t>0,$ $S_{v}(t)$ and $S_{-v}(t)$ have a basis of eigenvectors $
\{X_{i}(t):i=1,...,n-1\}$ and $\{\left. \text{d}s_{m}\right| _{\gamma
_{v}(t)}X_{i}(t):i=1,...,n-1\},$ respectively, with the same associated
eigenvalues. This in turn means that $M$ is a $\frak{SC}$-space and a $
\frak{SP}$-space.
\end{proposition}

The following result characterizes those Riemannian manifolds which are $k$
-D'Atri for all $k=1,...,n-1.$ Thus, we complete Theorem 2.6 of \cite{D10}.

\begin{theorem}
\label{th:k-D'A is GC} $M$ is a $n$-dimensional D'Atri space of type $k$ for
all $k=1,\dots ,n-1$ if and only if $M$ is a ${\frak{SC}}$-space.
\end{theorem}

\begin{proof} \ We fix a real number $t>0$ and a unit vector $v\in
T_{m}M. $ The hypothesis $\sigma _{k}(v,t)=\sigma _{k}(-v,t)$ for all $k=1,...,n-1$ implies that both characteristic polynomial of $S_{v}(t)$ and $
S_{-v}(t)$ are coincident. That is,
\begin{equation*}
\det \left( \lambda \text{Id}-S_{v}(t)\right) =\det \left( \lambda \text{Id}
-S_{-v}(t)\right)
\end{equation*}
and consequently, $S_{v}(t)$ and $S_{-v}(t)$ have the same set of
eigenvalues, counted with multiplicities.

The converse is immediate since by definition of $\frak{SC}$-space, $\sigma
_{k}(v,t)=\sigma _{k}(-v,t)$ for any small $t>0$ (see the expression of
$\sigma _{k}(v,t)$ in terms of the eigenvalues given in Section \ref{sec:intro}).
\end{proof}

On the other hand, Characterization 1.3 of \cite{B-P-V} proved by
A.~J.~Ledger and L. Vanhecke in \cite{LV} describes the locally symmetric
spaces ($\nabla R=0)$ as those whose local geodesic symmetries preserve the
shape operators of small geodesic spheres. Thus, we need a stronger
condition to assure when a $k$-D'Atri space for all $k=1,...,n-1$ (or a $
\frak{SC}$-space) is locally symmetric. In fact, from Proposition \ref
{pro:Known}, we have the following consequence.

\begin{corollary}
\label{cor:SymSP} Let $M$ be a $k$-D'Atri space for all $k=1,...,n-1.$
Then, $M$ is locally symmetric if and only if for each small real $t>0,$ the
geodesic symmetries $s_{m}$ preserve a basis of eigenvectors and the
associated eigenvalues of the shape operators $S_{v}(t)$ of small geodesic
spheres centered at $m$ for all $m\in M.$ That is, $M$ is locally symmetric
if and only if $M$ is a $\frak{SP}$-space (or it is $\frak{P}$-space, by
Theorem 3.2 of \cite{B-P-V} since $M$ is real analytic).
\end{corollary}

Now, we focus our attention in some applications. It is known that the
notions of $1$-D'Atri (D'Atri) space and $2$-D'Atri space are equivalent
(see \cite{D10}) although it is still open the question if they are also
equivalent to the $k$-D'Atri space notion for each $k,$ $3\leq k\leq n-1$.
Note that this is the case in the class of D'Atri spaces of Iwasawa type as
we show in Corollary \ref{cor:Iwa3DAtri}. Now, we will solve this problem
for the case of $4$-dimensional homogeneous Riemannian spaces.

\begin{corollary}
\label{cor:4dimDAtri} Let $M$ be a $4$-dimensional homogeneous Riemannian
space. If $M$ is a D'Atri space, then $M$ is a $k$-D'Atri space for all $k=1,2,3$.
Conversely, if $M$ is a $k$-D'Atri space for some $k=1,2,3$, then $M$ is a D'Atri space. In
particular, $M$ is a D'Atri space if and only if $M$ is a $3$-D'Atri space.
\end{corollary}

\begin{proof}
By Proposition \ref{pro:2.2D10} of Section \ref{sec:k-DAtri},
we know that every $k$-D'Atri space for some $k\geq 1$, in
particular every $3$-D'Atri space, satisfies at least the two first odd
Ledger conditions. In \cite{AM-Ko.08} has been proved that every $4$
-dimensional homogeneous Riemannian space that satisfies the two first odd
Ledger conditions is necessary naturally reductive and consequently it is a
D'Atri space. Moreover, all of them are conmutative since $\dim M\leq 5$
(see \cite[p. 10]{B-Tr-V}).

On the other hand, in Proposition 4.7 and Proposition 4.8 of \cite{B-P-V}
was proved that every $4$-dimensional Riemannian space that is naturally
reductive is a $\frak{TC}$-space and that in the class of commutative
spaces, ${\frak{SC}}$-spaces and $\frak{TC}$-spaces form the same subclass.
Therefore, every $4$-dimensional homogeneous D'Atri space is a ${\frak{SC}}$
-space and by Theorem \ref{th:k-D'A is GC} it is a $k$-D'Atri space for all $
k=1,2,3$. In particular, every D'Atri space is a $3$-D'Atri space.
\end{proof}

Moreover, using the equivalence between the properties of $M$ being a D'Atri
($1$-D'Atri) space or a $2$-D'Atri space we also have,

\begin{corollary}
Let $M$ be a $4$-dimensional homogeneous Riemannian space. $M$ is a $k$-D'Atri space for all $k=1,2,3$ if and only if $M$ is a $3$-D'Atri space.
\end{corollary}

\begin{remark}
\emph{
4-dimensional homogeneous Riemannian spaces provide examples of $k$-D'Atri
spaces for all $k=1,2,3$ which are not symmetric and consequently, by
Corollary \ref{cor:SymSP} they are neither $\mathfrak{SP}$-spaces nor $
\mathfrak{P}$-spaces. See in \cite{AM-Ko.08} the case 2 of Proposition 2,
the cases 1 and 2 of Proposition 5 and the case 4 of Proposition 6. All of them belong to the case ii) of the Classification Theorem of \cite{AM-Ko.08}. Thus, all of them are locally isometric to a Riemannian product $M^3\times \mathbb{R}$, where $M^3$ is naturally reductive. Thus, these examples are locally isometric to naturally reductive homogeneous spaces.}
\end{remark}

We finish this section with the following proposition that relates $\frak{C}$
-spaces and geodesic symmetries that preserve eigenvalues of Jacobi
operators.

\begin{proposition}\label{pro:C-cha}
$M$ is a $\frak{C}$-space if and only if the geodesic symmetries $s_{m}$
preserve the eigenvalues of Jacobi operators $R_{\gamma _{v}^{\prime }(t)}$
along the geodesics $\gamma _{v}(t),$ for all $m\in M$ and all unit vector $
v\in T_{m}M.$
\end{proposition}

\begin{proof}
We fix $m\in M$ and let $v\in T_{m}M$ be a unit vector. Note
we consider $R_{\gamma _{v}^{\prime }(t)}$ as $\left. R_{\gamma _{v}^{\prime
}(t)}\right| _{\gamma _{v}^{\prime }(t)^{\perp }}$ whenever $\gamma _{v}(t)$
is defined. Assume that $s_{m}$ preserves the eigenvalues of $R_{\gamma
_{v}^{\prime }(t)},$ which means that $\left. R_{\gamma _{v}^{\prime
}(t)}\right| _{\gamma _{v}^{\prime }(t)^{\perp }}$and $\left. R_{\gamma
_{-v}^{\prime }(t)}\right| _{\gamma _{-v}^{\prime }(t)^{\perp }}=\left.
R_{\gamma _{v}^{\prime }(-t)}\right| _{\gamma _{v}^{\prime }(-t)^{\perp }}$
have the same eigenvalues $\lambda _{v}(t)=\lambda _{-v}(t),$ respectively,
at $\gamma _{v}(t)$ and $\gamma _{-v}(t)=\gamma _{v}(-t)$. Thus, for any
possible real number $t$%
\begin{equation*}
\text{tr}R_{\gamma _{v}^{\prime }(t)}^{k}=\text{tr}R_{\gamma _{-v}^{\prime
}(t)}^{k}\text{ for all natural }k\geq 1.
\end{equation*}
Taking derivatives, as functions of real $t,$ we get
\begin{equation*}
\text{tr}\left( R_{\gamma _{v}^{\prime }(t)}^{k-1}\circ R_{\gamma
_{v}^{\prime }(t)}^{\prime }\right) =\text{tr}\left( R_{\gamma _{-v}^{\prime
}(t)}^{k-1}\circ R_{\gamma _{-v}^{\prime }(t)}^{\prime }\right) \text{ for
all }k\geq 1,
\end{equation*}
that evaluated at $t=0$ gives
\begin{equation*}
\text{tr}\left( R_{v}^{k-1}\circ R_{v}^{\prime }\right) =\text{tr}\left(
R_{-v}^{k-1}\circ R_{-v}^{\prime }\right) =-\text{tr}\left( R_{v}^{k-1}\circ
R_{v}^{\prime }\right) \text{ for all }k\geq 1,
\end{equation*}
since $R_{-v}^{\prime }=-R_{v}^{\prime }.$ Hence,
\begin{equation*}
\text{tr}\left( R_{v}^{k-1}\circ R_{v}^{\prime }\right) =0\text{ for }
k=1,...,n-1.
\end{equation*}
This fact implies that $\lambda _{v}^{\prime }(0)=0$ for all unit vector $v.$
By considering the above equation at $\gamma _{v}^{\prime }(t)$ ($\left|
\gamma _{v}^{\prime }(t)\right| =1$), we have that $\lambda _{v}^{\prime
}(t)=0$ for all $t.\,$ Thus, the eigenvalues of $R_{\gamma _{v}^{\prime
}(t)} $ are constant functions of $t$ (see [15, Theorem 2.6] or [7, Theorem
3]); that is, $M$ is a $\frak{C}$-space.

The converse is immediate.
\end{proof}

\end{document}